\documentclass[11pt,a4paper]{amsart}

\usepackage{verbatim, amssymb}

\usepackage{verbatim,amssymb,hyperref}

  \theoremstyle{plain}
    \newtheorem{thm}{Theorem}[section]
    \newtheorem{prop}[thm]{Proposition}
    \newtheorem{cor}[thm]{Corollary}
    \newtheorem{lem}[thm]{Lemma}
  \theoremstyle{definition}
    \newtheorem{Def}[thm]{Definition}
    
    \newtheorem{notation}[thm]{Notation}
  \theoremstyle{remark}
    \newtheorem*{rem}{Remark}
    \newtheorem*{note}{Note}
    \newtheorem*{eg*}{Example}
\newcommand{\Z}{\mathbb{Z}}

\newcommand{\sym}[1]{\mathfrak{S}_{#1}}
\newcommand{\Lie}{\operatorname{Lie}}
\newcommand{\Liemax}{\operatorname{Lie^{\rm max}}}
\newcommand{\func}{\operatorname}
\newcommand{\Ind}{\func{Ind}}
\newcommand{\Res}{\func{Res}}

\begin{document}

\date{Nov 2009}

\title{The Lie module of the symmetric group}

\author{Karin Erdmann}
\address[K. Erdmann]{Mathematical Institute, 24--29 St Giles', Oxford, OX1 3LB, United Kingdom.}
\email{erdmann@maths.ox.ac.uk}

\author{Kai Meng Tan}
\address[K. M. Tan]{Department of Mathematics, National University of Singapore, 2, Science Drive 2, Singapore 117543.}
\email{tankm@nus.edu.sg}

\begin{abstract}
We provide an upper bound of the dimension of the maximal projective submodule of the Lie module of the symmetric group of $n$ letters in prime characteristic $p$, where $n = pk$ with $p \nmid k$.
\end{abstract}

\thanks{The second author thanks the Mathematical Institute, Oxford, for its hospitality during his visit in 2006, during which most of the work appearing here was done, and acknowledges support by MOE's Academic Research Fund R-146-000-089-112.}

\subjclass[2000]{20C30}

\maketitle

\section{Introduction}

The Lie module of the symmetric group $\sym{n}$ appears in many contexts; in particular it is
closely related to the free Lie algebra. One possible approach is to view it as
the right ideal of the group algebra $F\sym{n}$, generated by the
`Dynkin-Specht-Wever element'
$$\omega_n:= (1-c_2)(1-c_3)\dotsm (1-c_n)$$
where $c_k$ is the $k$-cycle $(1, 2, \dotsc, k)$.  We write
$\Lie(n) = \omega_nF\sym{n}$ for this Lie module.

It is well-known that
$\omega_n^2 = n\omega_n$, so if $n$ is non-zero in $F$ then $\Lie(n)$ is a direct
summand of the group algebra and hence is projective.
We are interested in this module when $F$ has prime characteristic $p$ and when
$p$ divides $n$. In this case $\omega_n$ is nilpotent, and
therefore $\Lie(n)$ always has non-projective summands, and
its module structure is not well-understood in general.

The module $\Lie(n)$ is not only part of the algebraic literature,
for example it has already appeared in the work of Witt \cite{Witt}
and Weber \cite{Weber},  it also occurs in
the context of algebraic topology, as a homology group in spaces
related to braid groups \cite{C}, configuration spaces and
in the study of
Goodwillie calculus \cite{AM}; and
in \cite{AK}, Arone and Kankaanrinta calculate
the homology of $\mathfrak{S}_n$ with coefficients in $\Lie(n)$.
They show that it is zero unless $n=p^k$, and that
$H_*(\mathfrak{S}_{p^k}, {\rm Lie}(p^k))$ has a basis corresponding
to admissible sequences of Steenrod operations of length $k$.

The main motivation for our paper comes from the work of Selick and Wu \cite{SW}.
Their problem is to find natural homotopy decompositions
of the loop suspension of a  $p$-torsion suspension where $p$ is a prime.
In  \cite{SW} it is proved that this problem
is equivalent to the algebraic problem of finding natural
coalgebra decompositions of the
primitively generated tensor algebras over the field with $p$ elements.
They determine the finest coalgebra decomposition of
a tensor algebra  (over arbitrary fields), which can be described
as a  functorial Poincar\`{e}-Birkhoff-Witt theorem
\cite[Theorem 6.5]{SW}.
In order to compute the factors in this decomposition, one must know
a maximal projective submodule, called $\Liemax(n)$, of the Lie module
${\rm Lie}(n)$.

Since $\Lie(n)$ is a finite-dimensional module, and
since projective modules for symmetric groups are injective, there
is a direct sum decomposition of the form
$\Lie(n) = \Lie(n)_{pr} \oplus \Lie(n)_{pf}$, unique up to
isomorphism, where $\Lie(n)_{pr}$ is projective and $\Lie(n)_{pf}$ does not have any non-zero projective summand. Then $\Liemax(n)$ is isomorphic to
$\Lie(n)_{pr}$. As mentioned above, if $p$ does not divide $n$ then
$\Lie(n) = \Lie(n)_{pr}$; on the other hand
whenever $p$ divides $n$, the summand $\Lie(n)_{pf}$ is non-zero
(though, the homology of the symmetric group with coefficients in
$\Lie(n)$, that is, with coefficients in $\Lie(n)_{pf}$, is zero
if $n$ is not a power of $p$).

The projective modules for the symmetric groups over fields of
positive characteristic are not known.  Their structure depends
on the decomposition matrices for symmetric groups, and the determination of the latter is a famous open problem.
According to \cite{SW2}, it would
be interesting to know, even if the modules
cannot be computed precisely, how quickly the dimensions grow, and
whether or not the growth rate is exponential. Evidence in \cite{SW2}, for small cases in characteristic $2$, is that $\Liemax(n)$ is relatively
large compared with $\Lie(n)$ and this would correspond to
factors in the functorial PBW theorem being relatively small.

When $n=pk$ and $p$ does not divide $k$,
a parametrisation of the
indecomposable summands of $\Lie(n)_{pf}$ was given in \cite{ES}.
Here we exploit this
result to obtain an upper bound for the dimension of $\Lie^{\max}(n)$. The general principle is
quite easy.  If $P$ is a subgroup of a finite group
$G$, then one may consider the restriction of any $FG$-module $W$ to $FP$, which we denote as
$\Res^G_P W$.  Then $\Res^G_P (W_{pr})$ is a direct summand of $(\Res^G_P W)_{pr}$ and therefore
$$\dim W_{pr} \leq \dim\, (\Res^G_P W)_{pr} \leq \dim W - \dim\, (\Res^G_P W)_{pf}.
$$
Thus, when $G = \sym{n}$ and $W=\Lie(n)$, we have
$$\dim \Lie^{\max}(n) \leq (n-1)! - \dim\, (\Res^{\sym{n}}_P \Lie(n))_{pf}.
$$
In this paper, we take $P$ to be a Sylow $p$-subgroup of $\sym{n}$, and we
provide in particular a recursive formula for computing $\dim\, (\Res^{\sym{n}}_P \Lie(n))_{pf}$.

We give  an outline of this paper. After introducing notation, we reduce
in Section \ref{S:main} the problem of computing $\dim (\Res^{\sym{kp}}_P \Lie(pk))_{pf}$ to that of computing the number of certain right cosets.  Section \ref{S:Sylow} studies the Sylow $p$-subgroup $P$ of $\sym{pk}$, especially its elements of order $p$ which are fixed-point-free as elements of $\sym{pk}$.
In Section \ref{S:exist}, we obtain a `good' subset containing a transversal of the right cosets which we wish to parametrise, and proceed to obtain a transversal in Section \ref{S:unique}.
We end the paper with a recursive formula for the dimension
of
$(\Res^{\sym{kp}}_P \Lie(pk))_{pf}$, and we show that while this
dimension grows exponentially
with $k$, its ratio to $\dim(\Lie(kp))$ approaches zero as $k$ tends to infinity.
At present, no parametrisation of
the indecomposable summands of $\Lie(n)$ is known when  $n$ is divisible
by $p^2$, therefore the problem of  finding bounds for
the dimension of $\Lie^{\rm max}(n)$ cannot be approached in the
same way.

We refer the reader to
\cite{B}
for the necessary background on representation theory of finite groups.

\section{Notations} \label{S:not}

In this section, we introduce the notations to be used throughout this paper.

For $a,b \in \mathbb{Z}_{\geq 0}$ with $a< b$, let
\begin{align*}
[a,b] &= \{ x \in \mathbb{Z} \mid a \leq x \leq b \}, \\
(a,b] &= \{ x \in \mathbb{Z} \mid a < x \leq b \}, \\
[a,b) &= \{ x \in \mathbb{Z} \mid a \leq x < b \}.
\end{align*}

%Let $G$ be a group and suppose it acts on a non-empty set $X$ (from the right).  For each $g \in G$, write
%\begin{align*}
%\fix(g) &= \{ x \in X \mid x g = x \}.
%\end{align*}
%If $(x_1,x_2,\dotsc, x_r)$ is a sequence of elements of $X$,
%let $G_{(x_1,x_2,\dotsc, x_r)} = \{ g \in G \mid x_i \in \fix(g)\text{ for all } i \}$,
%the pointwise stabiliser of the set $\{x_1,x_2,\dotsc, x_r \}$.
We denote by $\sym{n}$ the group of permutations
of the set $[1, n]$.
For $m < n$, we identify $\sym{m}$ with the subgroup of $\sym{n}$ fixing $(m, n]$ pointwise.

Let $a,b \in \mathbb{Z}^+$.
For $\sigma \in \sym{a}$, define
$\sigma^{[b]} \in \sym{ab}$ by
$$((i-1)b+  j) \sigma^{[b]} = (i\sigma-1)b +j
$$
for all $i \in [1,a]$ and $j \in [1,b]$, so that $\sigma^{[b]}$
permutes the $a$ successive blocks of size $b$ in $[1, ab]$
according to $\sigma$.
Clearly, the map $\sigma \mapsto \sigma^{[b]}$ is an injective group homomorphism.

For $\tau \in \sym{b}$ and $r \in [1,s]$, define $\tau[r] \in \sym{sb}$ by
$$
((i-1)b + j) \ \tau[r] =
\begin{cases}
(r-1)b + j\tau, &\text{if } i = r, \\
(i-1)b + j, &\text{otherwise},
\end{cases}
$$
for all $i \in [1,s]$ and $j \in [1,b]$, so that $\tau[r]$ acts on the $r$-th successive block of size $b$ in $[1, sb]$ according to $\tau$, and fixes everything else.  Note that as a permutation on the set $((r-1)b, rb]$, $\tau[r]$ is independent of $s$ (as long as $s \geq r$).
As this notation also depends on $b$ (i.e.\ the degree of the symmetric group in which $\tau$ lies), we will specify $b$ when it is unclear from the context what $b$ is.

In addition, define $\Delta_a \tau \in \sym{ab}$ by $\Delta_a \tau = \prod_{i=1}^{a} \tau[i]$, so that $\Delta_a \tau$ permutes each of the $a$ successive blocks of size $b$ in $[1, ab]$ simultaneously according to $\tau$.  Clearly, the maps $\tau \mapsto \tau[r]$ and $\tau \mapsto \Delta_a\tau$ are injective group homomorphisms.

If $H \subseteq \sym{a}$, $K \subseteq \sym{b}$ and $r \in \mathbb{Z}^+$, we write
\begin{align*}
H^{[b]} &= \{ h^{[b]} \mid h \in H \} ,\\
K[r] &= \{ k[r] \mid k \in K \} ,\\
\Delta_a K &= \{ \Delta_a k \mid k \in K \}.
\end{align*}

We note the following lemma, whose proof is straightforward.

\begin{lem} \label{L:ready}
Let $\sigma \in \sym{a}$ and $\tau \in \sym{b}$.
\begin{enumerate}
\item If $r \in [1,a]$, then $\tau[r]^{\sigma^{[b]}} = \tau[r \sigma]$.
\item $\sigma^{[b]}(\Delta_a \tau) = (\Delta_a \tau) \sigma^{[b]}$.
\end{enumerate}
\end{lem}

Assume $n=pk$. Given a set partition of $[1, n]$ into $k$
blocks of size $p$, we have a subgroup $D$ of $\sym{n}$ which is isomorphic
to $\sym{p}\times \sym{k}$ where each factor acts regularly: the factor $\sym{k}$ permutes the $k$ blocks according to $\sym{k}$, while the factor
$\sym{p}$ act on each of the $k$ blocks
simultaneously, and these two actions commute.
Any two such groups arising from two set partitions are conjugate in $\sym{n}$.

Such a group $D$
can be viewed as a subgroup of a wreath product $\sym{p}\wr \sym{k}$:  let $\sym{k}^{[p]}$ be the fixed top group; its centralizer in
the base group is the diagonal product $\Delta_k\sym{p}$ which is isomorphic
to $\sym{p}$, and then one can take  $D=\Delta_k\sym{p}\times \sym{k}^{[p]}$.

%One can equally well take $D$ to be
%the subgroup $\Delta_p\sym{k}\times \sym{p}^{[k]}$ of $\sym{k}\wr \sym{p}$ with the analogous construction.

\begin{comment}
Let $p$ be a fixed prime integer, and let $\pi_0 = (1,2,\dotsc, p) \in \sym{p}$.
For each $i \in \mathbb{Z}^+$, define $\pi_i \in \sym{p^{i+1}}$ by $\pi_i = \pi_0^{[p^i]}$.

The following Lemma can be readily verified.

\begin{lem} \hfill
\begin{enumerate}
\item $\pi_i = \prod_{j=1}^{p^i} (j, p^i+j, 2p^i+j, \dotsc, (p-1)p^i + j)$.
\item If $\sigma \in \sym{p^i}$, then $\sigma^{\pi_{i}^j} = \sigma_{\bar{j}}$, where $1\leq \bar{j} \leq p$ and $j \equiv i \pmod p$.
\end{enumerate}
\end{lem}
\end{comment}

\section{The problem} \label{S:main}

Assume $F$ is a field of characteristic $p$, and let $k \in \mathbb{Z}^+$ with $p \nmid k$ and let $n=pk$.
%Let $D\leq \sym{n}$
%be a direct product of $\sym{p}$ with $\sym{k}$ where each factor
%acts regularly, as described above in Section \ref{S:not}.
In \cite[\S 4]{ES}, the module $\Lie(n)$ is studied via
a different module, called
the $p$-th symmetrisation of $\Lie(k)$, denoted by $S^p(\Lie(k))$; we will give
a precise definition below.
This module is related to $\Lie(n)$ as follows.

\begin{thm}{\cite[Theorem 10]{ES}} Assume  $n=pk$ where
$p$ does not divide $k$. Then
there is a short exact sequence of right $F\sym{n}$-modules
$$
0 \to \Lie(n) \to e F\sym{n} \to S^p(\Lie(k)) \to 0
$$
where $e$ is an idempotent in $\sym{n}$.
\end{thm}

As a corollary, we see that $\Omega(S^p(\Lie(k))) \cong \Lie(n)_{pf}$
(where here, and hereafter, $\Omega$ is the Heller operator, taking a module to the kernel
of its projective cover).

To define $S^p(\Lie(k))$, take the regular subgroup $D = \Delta_k(\sym{p})
\times \sym{k}^{[p]}$ as described at the end of \S 2.  Let
$\Lambda_k$ be the outer tensor product $\Lambda_k = F\boxtimes \Lie(k)$,
and take $S^p(\Lie(k)) \cong \Ind_D^{\sym{n}} \Lambda_k$.
[The module is the same as that given in \cite{ES}].
The action of $\Delta_k\sym{p}$ on $\Lambda_k$ is trivial, while that of
$\sym{k}^{[p]}$ on $\Lambda_k$ is equivalent to that of $\sym{k}$ on $\Lie(k)$.

%\noindent
Let $P$ be a fixed Sylow $p$-subgroup of $\sym{n}$.  By Mackey's formula, we have
\begin{align*}
\Res^{\sym{n}}_P S^p(\Lie(k)) &\cong \Res^{\sym{n}}_P \Ind_D^{\sym{n}} \Lambda_k
= \bigoplus_{x \in D / \sym{n} \setminus P} \Ind_{D^x \cap P}^P (\Lambda_k \otimes x).
\end{align*}

\bigskip

\begin{prop} \label{P:iff} \hfill
\begin{enumerate}
\item If $(\Delta_k\sym{p})^x \cap P = 1$, then $\Ind_{D^x \cap P}^P (\Lambda_k \otimes x)$ is projective.
\item If $(\Delta_k\sym{p})^x \cap P \ne 1$, then $\Ind_{D^x \cap P}^P (\Lambda_k \otimes x)$ has no projective summand.
\end{enumerate}
\end{prop}

\begin{proof}
If $(\Delta_k\sym{p})^x \cap P = 1$, then $\Delta_k\sym{p} \cap P^{x^{-1}} =1$; we claim that in this case $\Res_{D \cap P^{x^{-1}}} \Lambda_k$ is projective, so that (1) follows.

To prove the claim, let $Q = D \cap P^{x^{-1}}$, then
$\Delta_k\sym{p} \cap Q= 1$.
If $R$ is a Sylow $p$-subgroup of $\Delta_k\sym{p}$, then all $D$-conjugates of $R$ lie in $\Delta_k\sym{p}$, since $\Delta_k\sym{p}$ is normal in $D$; thus $R^d \cap Q = 1$ for all $d \in D$.
Now, $\Lambda_k$ is by construction relatively $R$-projective, so that $\Lambda_k$ is a direct summand of $\Ind^D_R U$ for some $R$-module $U$.  It follows that $\Res_Q \Lambda_k$ is a direct summand of $\Res_Q \Ind^D_R U$.  But by Mackey's formula, $\Res_Q \Ind^D_R U = \bigoplus_{d \in R / D \setminus Q} \Ind_{R^d \cap Q}^Q (U \otimes x)$.  Since $R^d \cap Q = 1$, each summand $\Ind_{R^d \cap Q}^Q (U \otimes x)$ is projective.  Thus, $\Res_Q \Ind^D_R U$ and $\Res_Q \Lambda_k$ are projective.

% so that the map $\Psi : D \to \sym{k}^{[p]}$ defined by $\Psi((\Delta_k \sigma) \tau^{[p]}) = \tau^{[p]}$ for $\sigma \in \sym{p}$ and $\tau \in \sym{k}$ is injective when restricted to $D \cap P^{x^{-1}}$.  As $\Psi(D\cap P^{x^{-1}})$ is a $p$-subgroup of $\sym{k}^{[p]}$, and $p \nmid k$, we have $\Psi(D \cap P^{x^{-1}}) \subseteq ((\sym{k})_i)^{[p]}$ for some $i \in \{1,2,\dotsc,k\}$.  Hence, $\Res_{\Psi(D\cap P^{x^{-1}})} \Lambda_k$ is free, and so $\Res_{D\cap P^{x^{-1}}} \Lambda_k$ is free since the action of $\Delta_k\sym{p}$ on $\Lambda_k$ is trivial (and $\Psi$ is injective when restricted to $D\cap P^{x^{-1}}$).  This implies that $\Ind_{D^x \cap P}^P (\Lambda_k \otimes x)$ is projective.

If $(\Delta_k\sym{p})^x \cap P \ne 1$, let $\sigma \in \sym{p}$ such that
$(\Delta_k\sigma)^x \in (\Delta_k\sym{p})^x \cap P$.
Then since $\Delta_k\sigma$ acts trivially on the entire module
$\Lambda_k$, we see that $\Ind_{D^x \cap P}^P (\Lambda_k \otimes x)$ cannot have any projective summand.
\end{proof}

In view of Proposition \ref{P:iff}, let $S$ be the set of all double
coset representatives in
$D / \sym{n}\setminus P$ such that $(\Delta_k\sym{p})^x\cap P\neq 1$.  Then we have

\begin{cor} \label{C:iff}
Assume $n = pk$ with $p \nmid k$.  Then
$$(\Res^{\sym{n}}_P S^p(\Lie(k)))_{pf}  \cong \bigoplus_{x \in S} \Ind_{D^x \cap P}^P (\Lambda_k \otimes x).$$
\end{cor}

The following identifies
the $\Omega$-translate of this module.

\begin{lem}\label{prepare-P:iff}
For $x\in S$ we have
$$\Omega(\Ind_{D^x\cap P}^P(\Lambda_k\otimes x))
\cong \Ind_{D^x\cap P}^P ((\Omega(F)\boxtimes \Lie(k))\otimes x),
$$ and
$\Omega(F)$ has dimension $p-1$.
\end{lem}

\begin{proof}
Recall that $\Lambda_k \cong F \boxtimes \Lie(k)$.
Since  $\Lie(k)$ is a projective $F\sym{k}$-module, we see that $\Omega(\Lambda_k) \cong \Omega(F) \boxtimes \Lie(k)$.  Furthermore, $\Omega(F)$ can be described as follows:
the natural $p$-dimensional permutation module of $F\sym{p}$ is indecomposable projective and has $F$ as a quotient, so that $\Omega(F)$ is its maximal submodule, of dimension $(p-1)$.
Moreover, $\Omega(F)$ remains indecomposable when restricted to any subgroup of $\sym{p}$ of order $p$.  Now, the short exact sequence
\begin{equation*}
0 \to \Omega(F) \boxtimes \Lie(k) \to P(\Lambda_k) \to \Lambda_k \to 0,
\end{equation*}
where $P(\Lambda_k)$ denotes the projective cover of $\Lambda_k$, gives the following short exact sequence
\begin{multline}
0 \to \Ind^P_{D^x \cap P} ((\Omega(F) \boxtimes \Lie(k)) \otimes x) \to \Ind^P_{D^x \cap P} (P(\Lambda_k) \otimes x) \\ \to \Ind^P_{D^x \cap P} (\Lambda_k \otimes x) \to 0.  \tag{$*$}
\end{multline}
Let $1\ne \sigma \in \sym{p}$ such that $
(\Delta_k\sigma)^x \in
(\Delta_k\sym{p})^x \cap P$.  Then $(\Delta_k\sigma)^x$ acts trivially on
$\Lie(k)$, and $\Omega(F)$ is indecomposable as a module for $\left<
\Delta_k\sigma  \right>$
(as $\Delta_k\sigma$ has order $p$) and has dimension $(p-1)$.  It follows that $\Ind^P_{D^x \cap P} ((\Omega(F) \boxtimes \Lie(k)) \otimes x)$ has no projective summand.
Thus, from $(*)$, we see that
$$\Omega(\Ind^P_{D^x \cap P} (\Lambda_k \otimes x)) \cong \Ind^P_{D^x \cap P} ((\Omega(F) \boxtimes \Lie(k)) \otimes x)$$
since $\Ind^P_{D^x \cap P} (P(\Lambda_k) \otimes x)$ is projective.
\end{proof}

The following is the main result of this section.

\begin{thm} \label{T:main}
Assume $n = pk$ with $p \nmid k$.
We have
$$\dim ((\Res^{\sym{n}}_P \Lie(n))_{pf}) = (p-1) (k-1)! \sum_{x\in S} [P:D^x \cap P].$$
\end{thm}

\begin{proof}
Restriction is  exact, hence we have a short exact sequence
$$
0 \to \Res_P \Lie(n) \to \Res_P (e F\sym{n}) \to \Res_P S^p(\Lie(k)) \to 0.$$
Since $\Res_P (e F\sym{n})$ remains projective as an $FP$-module, we see that
\begin{align*}
(\Res_P \Lie(n))_{pf} \cong \Omega(\Res_P S^p(\Lie(k))
%&\cong \Omega(\Res_P \Ind_D^{\sym{n}} \Lambda_k) \\
\cong \bigoplus_{x\in S} \Omega(\Ind_{D^x \cap P}^P (\Lambda_k \otimes x))
\end{align*}
by Corollary \ref{C:iff}. Now, by \ref{prepare-P:iff}, we
have
$\Omega(\Ind_{D^x\cap P}^P (\Lambda_k\otimes x))$ for $x\in S$
 is isomorphic to
$\Ind_{D^x\cap P}^P((\Omega(F)\boxtimes \Lie(k))\otimes x)$ and that
$\Omega(F)$ has dimension $p-1$.
This proves the Theorem.
\end{proof}

\begin{cor} \label{C:main}
We have
$$\dim ((\Res^{\sym{kp}}_P \Lie(kp))_{pf}) = (p-1) (k-1)! N,$$
where $N$ is the number of cosets $Dx$ such that $(\Delta_k\sym{p})^x \cap P \ne 1$.
\end{cor}

\begin{proof}
Using Theorem \ref{T:main} we must show that $N = N'$ where
$N'= \sum_{y\in S} [P:D^y\cap P]$.
The index $[P:D^y\cap P]$ is equal to the size of the $P$-orbit
of $Dy$ in the coset space $(\sym{n}:D)$, so it is equal to the number of cosets $Dx$ contained in
$DyP$.

Write $D_1 = \Delta_k\sym{p}$.
If a coset $Dx$ is contained in $DyP$ then $D_1^x\cap P$ is $P$-conjugate
to $D_1^y\cap P$ and hence  $D_1^y\cap P\neq 1$ if and only if
$D_1^x\cap P\neq 1$. Conversely if $Dx$ is a coset and $D_1^x\cap P\neq 1$ then $Dx$ is contained in one of the double cosets counted for $N'$.
We sum over all such
double cosets,  and hence
$N= N'$.
\end{proof}

\medskip

\medskip

%Let $S = \{ x \in \sym{pk} \mid (\sym{p}^{[k]})^x \cap P \ne 1 \}$.  It suffices to show that
%$$N = \sum_{x \in (D/\sym{pk}\setminus P) \cap S} (P:D^x \cap P).$$
%Note that $S$ admits a natural left $D$-action and a natural right $P$-action.  These two commuting actions induces a right $P$-action on the $D$-orbits which are the right cosets $Dx$.  The stabiliser of $Dx$ is precisely $D^x \cap P$, so that the index $(P:D^x\cap P)$ is the size of the $P$-orbit containing $Dx$.  As each double coset $DxP$ corresponds naturally to the $P$-orbit containing $Dx$, we see that $\sum_{x \in (D/\sym{pk} \setminus P) \cap S} (P:D^x \cap P)$ is exactly $N$, the total number of $Dx$ where $x \in S$.
%\end{proof}

Corollary \ref{C:main} suggests that we should proceed by parametrising the right cosets $Dx$ such that $(\Delta_k \sym{p})^x\cap P\neq 1$.
When $(\Delta_k \sym{p})^x\cap P\neq 1$, this is a non-trivial $p$-subgroup of $P$ conjugate to a subgroup of
$\Delta_k \sym{p}$, and hence it is generated by an element $y \in P$ of order $p$ which is fixed-point-free as an element of $\sym{kp}$.  We shall study these elements in the next section.

\section{Sylow $p$-subgroups of symmetric groups} \label{S:Sylow}

We analyze the Sylow $p$-subgroups of symmetric groups, especially its fixed-point-free elements of order $p$, in this section.
In order to exploit the orbit structure of such elements,
we make the following definitions.

\begin{Def}
Let $n \in \mathbb{Z}^+$.
\begin{enumerate}
\item A {\em partition} of the interval $(0,n]$ is an increasingly ordered subset $\{a_0,\dotsc, a_s\}$ of $(0,n]$ with $a_0 = 0$ and $a_s = n$.
\item A partition $\{b_0,\dotsc, b_t\}$ of $(0,n]$ is {\em finer} than $\{a_0,\dotsc, a_s\}$ provided for each $r \in [1,s]$, there exists $j_r \in [1,t]$ such that $a_r = b_{j_r}$.
\item Let $\sigma \in \sym{n}$.
The partition $\{a_0,\dotsc, a_s\}$ of $(0,n]$ is {\em $\sigma$-invariant} provided for each $r \in [1,s]$, the subinterval $(a_{r-1},a_r]$ is $\sigma$-invariant, that is \ $\sigma(a_{r-1},a_r] = (a_{r-1},a_r]$.
\end{enumerate}
\end{Def}

\begin{lem} \label{L:finest}
Let $\sigma \in \sym{n}$.  There is a unique finest $\sigma$-invariant partition of the interval $(0.n]$.
\end{lem}

We may thus speak of {\em the} finest $\sigma$-invariant partition of $(0,n]$, which is finer than every other $\sigma$-invariant partition of $(0,n]$.
We leave the proof of the Lemma to the reader.

%\begin{proof}
%Note that a partition $\{a_0,\dotsc, a_s\}$ of $(0,n]$ is $\sigma$-invariant if and only if for each $r \in [1,s]$, $(a_{r-1},a_r]$ is a union of some $\left< \sigma \right>$-orbits.  It is clear that finest $\sigma$-invariant partition of $(0,n]$ exists.  Suppose for the sake of contradiction that $\{a_0,\dotsc, a_s\}$ and $\{b_0,\dotsc, b_t\}$ are two distinct finest partition of $(0,n]$, say $a_r < b_r$, and $a_u = b_u$ for all $u \in (0,r)$.  Then $(b_{r-1}, a_r]$ and $(b_{r-1},b_r]$ are both union of some $\left< \sigma \right>$-orbits, so that $(a_r,b_r]$ is also a union of some $\left< \sigma \right>$-orbits.  But then $\{ b_0,\dotsc, b_{r-1}, a_r, b_r,\dotsc, b_t\}$ is a partition of $(0,n]$, and is finer than $\{b_0,\dotsc, b_t\}$, a contradiction.
%\end{proof}

The Sylow $p$-subgroups of $\sym{n}$ are direct products of iterated wreath products of cyclic groups of order $p$.  We proceed with the analysis of these building blocks.

%\noindent
From now on, we denote the distinguished $p$-cycle $(1,2,\dotsc,p)$  by $\pi$.

\begin{Def} \hfill
\begin{enumerate}
\item For each $i \in \mathbb{Z}_{\geq 0}$, let $R_i$ be the subgroup of $\sym{p^{i+1}}$ generated by $\{ \pi^{[p^j]} \mid j \in [0,i] \}$.
Then $R_i$ is a Sylow $p$-subgroup of $\sym{p^{i+1}}$.
 For convenience, let $\pi^{[p^{-1}]} = 1$, and $R_{-1} = 1$.
\item For each $i \in \mathbb{Z}_{\geq 0}$, let $B_i = R_{i-1}[1] \times R_{i-1}[2] \times \dotsb \times R_{i-1}[p]$.  For convenience, let $B_{-1} = \emptyset$.
\item And for each $s \in \mathbb{Z}^+$, let $H_s = \prod_{a=1}^s \left< \pi[a] \right>$.
\end{enumerate}
\end{Def}

\begin{note}
For each $i \in \mathbb{Z}_{\geq 0}$ we have $R_i = B_i \rtimes \langle \pi^{[p^i]} \rangle$,
and, as a group, $R_i$ is isomorphic
to $\smash{\underbrace{C_p \wr C_p \wr \dotsb \wr C_p}_{(i+1) \text{ times}}}$.
Also, $H_{p^i}$ is a subgroup of $B_i$, and it is normal in $R_i$.
\end{note}

The finest $g$-invariant partition of the interval $(0,p^{i+1}]$ for $g \in R_i$ has nice properties:

\begin{prop} \label{P:wreath}
Suppose that $g \in R_i$, and let $\{a_0 < a_1 < \dotsb < a_s \}$ be the finest $g$-invariant partition of the interval $(0,p^{i+1}]$.
Then
\begin{enumerate}
\item $a_j- a_{j-1}$ is a power of $p$, dividing $a_{j}$, for all $1 \leq j \leq s$;
\item  $g = \prod_{j=1}^s \gamma_j[a_j/p^{\lambda_j}]$
where $a_j - a_{j-1} = p^{\lambda_j}$,  with $\gamma_j \in R_{\lambda_j-1}\setminus B_{\lambda_j-1}$ for all $1 \leq j \leq s$.
\end{enumerate}
\end{prop}

Note that the factor $\gamma_j[a_j/p^{\lambda_j}]$ appearing in Proposition \ref{P:wreath}(2) is just $\gamma_j$ with its support translated from $(0,p^{\lambda_j}]$ ($= (0,a_j-a_{j-1}]$) to $(a_{j-1},a_j]$.

\begin{proof}
We prove the first two statements by induction on $i$.
When $i=0$, either $g = 1$,
in which case $s = p$ and $a_j = j$ for all $0 \leq j \leq s$, or else $g = \pi^t$
for some $t \in \mathbb{Z}_p^*$, in which case $s =1$ and so $a_1 = p$.
It can easily be checked that both statements hold in either case.

Assume now $i>0$.
The statements are trivial when $s=1$.  When $s>1$, we have $g \in B_i$.
%, as
%otherwise $g \in B_i (\pi^{[p^i]})^t$ for some $t \in \mathbb{Z}_p^*$,
%and then $g$ would not stabilise $(0, a_1]$.
Thus, $g = \prod_{j=1}^p g_j[j]$ for some $g_1,g_2,\dotsc, g_p \in R_{i-1}$, so that $\{ 0, p^i, 2p^i,\dotsc, p^{i+1}\}$ is a $g$-invariant partition of $(0,p^{i+1}]$.
Hence by Lemma \ref{L:finest}, there exist $j_1, \dotsc, j_p \in  [1,s]$ such that $a_{j_r}= rp^i$ for
$r \in [1,p]$.
Now the partition $\{a_{j_{r-1}}, \dotsc, a_{j_r}\}$ of the interval $((r-1)p^i, rp^i]$ must be the finest $g_r$-invariant partition of $(0,p^i]$ linearly translated by $(r-1)p^i$.  Induction hypothesis applied to $g_r$ then yields for each $j \in (j_{r-1},j_r]$, $a_j - a_{j-1}$ is a power of $p$, say $p^{\lambda_j}$, dividing $a_j - (r-1)p^i$.  Since $p^{\lambda_j} = a_j - a_{j-1} \leq p^i$, we see that $\lambda_j \leq i$, and so $p^{\lambda^j}$ divides $(a_j - (r-1)p^i) + (r-1)p^i = a_j$.  Also by induction, $g_r$ is a product of $\gamma_j \in R_{\lambda_j-1} \setminus B_{\lambda_j-1}$ whose support is translated from $(0,p^{\lambda_j}]$ to $(a_{j-1} - (r-1)p^i, a_j - (r-1)p^i]$.  Thus $g_r[r]$ is a product of $\gamma_{j}$ whose support is translated from $(0,p^{\lambda_j}]$ to $(a_{j-1},a_j]$, and the proof is complete.
\end{proof}

For $g \in R_i$ of order $p$ and is fixed-point-free as an element of $\sym{p^{i+1}}$, the $\lambda_j$'s appearing in Proposition \ref{P:wreath}(2) are all positive.  Furthermore, each $\gamma_j$ is conjugate to a unique power of $\pi^{[p^i]}$; we present below a proof of this, in a more general setting.

\begin{prop} \label{P:conj}
Let $G$ be a group of the
form $G = R\wr \langle y \rangle$ where $y$ has order $p$, with
base group $B$.  Let $t \in \mathbb{Z}_p^*$.  Then the conjugacy class of $y^t$ contains precisely the elements of $G$ having order $p$ and lying in the coset $By^t$.
\end{prop}

Thus, the elements of $G \setminus B$ having order $p$ are just the various conjugates of $y^t$ for $t \in\mathbb{Z}_p^*$.

%Let $\Gamma = \{ g \in G \setminus B \mid g \text{ has order }p \}$, and for each $t \in \mathbb{Z}_p^*$, let $C_t$ be %the conjugacy class of $y^t$, and $\Gamma_t = \Gamma \cap By^t$.  Then
%\begin{enumerate}
%\item $\Gamma_t = C_t$;
%\item if $g \in C_t$, then $g = (b')^{-1}y^tb'$ for a unique $b' \in B'$, where $B'$ be the direct product of $p-1$ copies of $R$, say $B' = R \times \dotsb \times R \times 1 \subseteq B$.
%\end{enumerate}
%Thus, the elements of $G \setminus B$ of order $p$ are precisely the conjugates of the $y^t$'s for $t \in \mathbb{Z}_p^*$.
%
%Then the elements of $P\setminus B$ of order $p$ are precisely the conjugates
%of the $y^t$ for $t\in \mathbb{Z}_p^*$.  More precisely, let $B'$ be the direct product of $p-1$ copies of $R$, say $B' %= R \times \dotsb \times R \times 1 \subseteq B$; then the conjugacy class of $y^t$ is the set
%$$\{ (b')^{-1}y^tb': \ b'\in B'\}.$$
%In particular, every element of $P \setminus B$ of order $p$ has a unique expression as $(b')^{-1}y^tb'$ with $b' \in %B'$ and $t \in \mathbb{Z}_p^*$.
%\end{prop}

\begin{proof}
It suffices to prove the proposition for $t = 1$ since $y^t$ also generates the group $\left< y \right>$.  Let $C$ be the conjugacy class of $y$, and let $\Gamma = \{ g \in By \mid g$ has order $p \}$.  It is easy to see that $C \subseteq \Gamma$.  For the converse, we show that every element of $\Gamma$ is conjugate to $y$ by a unique element of $B'$, where $B'$ is the direct product of $p-1$ copies of $R$, say $B' = R \times \dotsb \times R \times 1 \subseteq B$.
Let $g = by \in \Gamma$, where $b \in B$.
Then $g^p = b(^yb)(^{y^2}b) \cdots  (^{y^{p-1}}b)$ (where
$^xb= xbx^{-1}$). If $b = (r_1, \ldots, r_p)$ then the coordinates of $g^p$ are the cyclic permutations of
$r_1r_2\ldots r_p$.
Hence $g^p=1$ if and only if $r_p = (r_1r_2\ldots r_{p-1})^{-1}$.
This shows that given $r_1, \ldots, r_{p-1} \in R$ there is a unique such $g$ of order $p$. Hence $|\Gamma| \leq |R|^{p-1}$.
The set $A:= \{ (b')^{-1}y b': b'\in B'\}$ is contained in $\Gamma$, and has size $|B'| = |R|^{p-1}$ since the centraliser of $y$ in $B'$ is trivial. It follows that $A = \Gamma$, and $g = (b')^{-1}y b'$ for a unique $b' \in B'$, and the proof is complete.
\end{proof}

\begin{cor}[of proof] \label{C:conj}
Every element of $R_i \setminus B_i$ having order $p$ can be uniquely expressed as $((\pi^{[p^i]})^t)^b$ with $b \in \prod_{j=1}^{p-1} R_{i-1}[j]$ and $t \in \mathbb{Z}_p^*$.
%Keep the notations in Proposition \ref{P:conj} and its proof.  Every element of $G \setminus B$ having order $p$ can be uniquely expressed as $(y^t)^{b}$ with $b \in B'$ and $t \in \mathbb{Z}_p^*$.
\end{cor}

%As a next step, we present a `normal form' for elements in  $R_i$.  Note that for $g \in R_i$, the partition $\{0,p^{i+1}\}$ of of the interval $(0,p^{i+1}]$ is the finest $g$-invariant partition if and only if $g \notin B_i$.
%Let $g\in \sym{n}$.  We say that a subinterval $(a,b]$ of $(0,n]$ is {\em $g$-irreducible} if and only if $g$ stabilises $(a,b]$ but does not stabilize $(a, a']$ for any $a'$ with $a<a'<b$. we also say $g$ acts irreducibly on $(a,b]$.
%For example, if $g=\pi^{[p^j]}$ then the interval $(0, p^{j+1}]$ is
%$g$-irreducible.
%On the other hand, an element in $B_j$ does not act irreducibly on
%$(0, p^{j+1}]$.

%\begin{eg} \label{ex}
%Let $p=3$, and let $g \in B_2 \subseteq R_2 = \langle \pi, \pi^{[3]}, \pi^{[3^2]}\rangle$, with
%$g = \prod_{i=1}^3 g_i[i]
%$
%where $g_1 = \pi^{[3]}$,
%$g_2$ is the product of three disjoint non-trivial powers of $\pi$, appropriately
%shifted, i.e.\ $g_2 = \pi^{t_1}[1]\pi^{t_2}[2]\pi^{t_3}[3]$ with $t_i \in \mathbb{Z}_3^*$ for all $i$, and $g_3= %(\pi^{[3]})^2$.
%Then
%$$a_0=0, \  a_1=9, \ a_2 = 9+3, \ a_3 = 9+6, \ a_4 = 2\cdot 9, \ a_5=27.
%$$
%The normal form for $g$ is then
%$g=\pi^{[3]}[1]\cdot\pi^{t_1}[3+1]\cdot\pi^{t_2}[3+2]\cdot \pi^{t_3}[2\cdot 3]\cdot (\pi^{[3]})^2[3].
%$
%\end{eg}

%\begin{rem}
%In the proof of Proposition \ref{P:conj}, we actually showed that if $g \in R_i \setminus B_i$ for some $i \in \mathbb{N}_0$, and $g$ has order $p$, then there exist a {\em unique} $b \in R_{i-1}[1] \times R_{i-1}[2] \times \dotsb \times R_{i-1}[p-1]$ and a unique $t \in \mathbb{Z}_p^*$ such that $g = ((\pi^{[p^i]})^t)^b$.
%\end{rem}

Take $g\in R_i$.  We have seen in
Proposition \ref{P:wreath}(1) that if $\{a_0,\dotsc,a_s\}$ is finest $g$-invariant partition of $(0,p^{i+1}]$, then $a_j - a_{j-1}$ are $p$-powers, dividing $a_j$.  Clearly, these $p$-powers, or more simply their exponents, completely determines the finest $g$-invariant partition.
%In Example \ref{ex}, these exponents are
%$(2, 1,1,1,2).$
This suggests the following terminology.

\begin{Def}\label{D:pcomp}
A {\em $p$-composition $\lambda = (\lambda_1,\lambda_2,\dotsc,\lambda_s)$} is a finite sequence of non-negative integers such that for each $j \in [1, s]$, the partial sum $\sum_{i=1}^j p^{\lambda_i}$ is divisible by $p^{\lambda_j}$.

Let $\lambda= (\lambda_1,\lambda_2,\dotsc,\lambda_s)$ be a $p$-composition.
%Write $l(\lambda) = s$, and
For each $j \in [1,s]$, we will denote these  partial sums in the following
 by
$$\Sigma^{\lambda}_j: = \sum_{i=1}^{j} p^{\lambda_i}.
$$
If $\sum_{i=1}^s p^{\lambda_i} = r$, we say that $\lambda$ is a {\em $p$-composition of $r$}.
\end{Def}

Clearly, if $\lambda = (\lambda_1,\lambda_2,\dotsc,\lambda_s)$ is a $p$-composition, then $\{ 0, \Sigma^{\lambda}_1, \Sigma^{\lambda}_2,\dotsc, \Sigma^{\lambda}_s \}$ is a partition of the interval $(0,\Sigma^{\lambda}_s]$.

%The above example $\lambda = (2,1,1,1,2)$ is a 3-composition.  More generally, we have the following.

\begin{prop}\label{P:p-comp-1}
Let $g \in R_i$, and let $\lambda = (\lambda_1, \lambda_2, \ldots, \lambda_s)$, where $\lambda_i$'s are as in Proposition \ref{P:wreath}(2).  Then $\lambda$ is a $p$-composition of $p^{i+1}$.

Conversely, every $p$-composition of $p^{i+1}$ arises in this way.  More precisely, if $\mu = (\mu_1,\dotsc, \mu_t)$ is a $p$-composition of $p^{i+1}$, then $\{ 0, \Sigma^{\mu}_1,\dotsc, \Sigma^{\mu}_t \}$ is the finest $g$-invariant partition of the interval $(0,p^{i+1}]$, where
$$g := \prod_{r=1}^t \pi^{[p^{\mu_r-1}]}[\Sigma^{\mu}_r/p^{\mu_r}] \in R_i.$$
\end{prop}

\begin{proof}
The first assertion follows immediately from Proposition \ref{P:wreath}, while the second can be verified easily.
\end{proof}

An important $p$-composition is defined as follows.

\begin{Def}\label{D:padic}
Let $r \in \mathbb{Z}^+$, and let $r = \sum_{i=0}^t b_i p^{i}$ be its $p$-adic expansion, with $b_t\neq 0$.
The $p$-composition $(t^{b_t}, (t-1)^{b_{t-1}}, \dotsc, 0^{b_0})$
with $b_t$ parts equal to $t$, $b_{t-1}$ parts equal to $t-1$, and so on,
is called the {\em $p$-adic $p$-composition of $r$}.
\end{Def}

\begin{lem}\label{L:refinement} \hfill
\begin{enumerate}
\item Let $\gamma$ be a $p$-composition of $ap^m$ for some $a \in \mathbb{Z}^+$,
 and let $\delta$ be a $p$-composition of $p^j$ with $j \leq m$. Then
the concatenation $(\gamma, \delta)$ is a $p$-composition.

\item If $\lambda =
(\lambda_1, \ldots, \lambda_s)$ is a $p$-composition, and $\mu^{(i)}$ is a $p$-composition of $p^{\lambda_i}$ for all $1\leq i\leq s$,
then the concatenation $\mu := (\mu^{(1)}, \mu^{(2)}, \ldots, \mu^{(s)})$ is
a $p$-composition.
\end{enumerate}
\end{lem}

\begin{proof}
For part (1), if $\delta_i$ is a part of $\delta$ then $\delta_i\leq j\leq m$
and therefore $p^{\delta_i}$ divides $ap^m$. Since $p^{\delta_i}$ also divides $\Sigma_i^{\delta}$ (as $\delta$ is a $p$-composition), it divides
$ap^m + \Sigma_i^{\delta}$. Thus
$(\gamma, \delta)$ is a $p$-composition.
Part (2) follows by induction and part (1).
\end{proof}

\begin{Def}\label{D:refinement} Suppose that $\lambda, \mu$ are as in Lemma
\ref{L:refinement}(2).  Then
we say that $\mu$ is a {\it refinement} of $\lambda$.
\end{Def}

%\bigskip

\begin{cor}\label{C:refinement} Let $\lambda$ be a $p$-composition of $p^{m}$ with more than
one part. Then $\lambda$ is a refinement of the $p$-composition $((m-1)^p)$.
\end{cor}

\begin{proof}
By Proposition \ref{P:p-comp-1}, $\lambda$ arises from the finest $g$-invariant partition $\mathcal{P}$ of the interval $(0,p^m]$ for some $g \in R_{m-1}$.  Since $\lambda$ has more than one part, so does $\mathcal{P}$.  Thus $g \in B_{m-1}$, so that $\{0,p^{m-1}, 2p^{m-1}, \dotsc, p^m\}$ is a $g$-invariant partition of $(0,p^m]$.  Hence $\mathcal{P}$ is finer than $\{0,p^{m-1}, 2p^{m-1}, \dotsc, p^m\}$, and in turn, $\lambda$ is a refinement of $((m-1)^p)$.
\end{proof}

Recall that $R_i$ is an explicit Sylow $p$-subgroup of $\sym{p^{i+1}}$,
with support $(0,p^{i+1}]$.  We will now fix an explicit Sylow $p$-subgroup
of $\sym{pk}$.  Let $\kappa = (\kappa_1,\dotsc, \kappa_l)$ be the $p$-adic $p$-composition of $k$.  Then $(0,p\Sigma^{\kappa}_1,\dotsc, p\Sigma^{\kappa}_l]$ is a partition of the interval $(0,pk]$.  We choose our Sylow $p$-subgroup to be the product of $R_{\kappa_i}$ whose support is translated from $(0,p^{\kappa_i+1}]$ to $(p\Sigma^{\kappa}_{i-1}, p\Sigma^{\kappa}_{i}]$.

%Suppose  $k$ has $p$-adic expansion $k = \sum_{i=0}^t k_ip^i$, as
%above, with $k_i \in [0, p)$, and where
%$k_t\neq 0$. Then a Sylow $p$-subgroup of $\sym{pk}$ is isomorphic to the direct products
%of $k_i$ copies of $R_i$ for $i \in [0,t]$, with disjoint supports.  We choose our Sylow $p$-subgroup $P_k$ of %$\sym{pk}$ by taking first
%$k_t$ copies of
%$R_t$, then $k_{t-1}$ copies of $R_{t-1}$, and so on.

%\begin{eg*}
%Let $n = pk=p^4 + 2p^3 + p$ with $p\geq 3$. We take $P$
%with orbits
%$$[1, p^4], \ \ (p^4, p^4+p^3], \ \ (p^4+p^3, p^4+2p^3],
%\ \ (p^4+2p^3, p^4+2p^3+p].
%$$
%The first factor of $P$ is $R_3$. The second factor is a shifted copy of $R_2$, acting on the next possible block of %size $p^3$. The support of the first factor $R_3$ is a $p$-block of
%size $p^4$;  the second factor is $R_2[p+1]$, with support $(p^4,p^4+p^3]$.
%Repeating this argument we get
%$$P = R_3[1] \times R_2[p+1]\times R_2[p+2] \times R_1[p^3+2p^2+1].
%$$
%Notice that the numbers $p^4$, $p^4+p^3$, $p^4+2p^3$, $p^4+2p^3+p$ are precisely the partial sums
%sums $p\Sigma_i^{\kappa}$ of the $p$-adic $p$-composition $\kappa = (3,2,2,0)$ of $k$.
%The support of the second factor is shifted by $\Sigma_2^{\kappa}/p^{\kappa_2}$, and the support
%of the third factor is shifted by $\Sigma_3^{\kappa}/p^{\kappa_3}$ (note that
%$\kappa_2=\kappa_3=2$). Finally the support of the fourth factor is shifted by
%$\Sigma_4^{\kappa}/p^{\kappa_4}= \Sigma_4^{\kappa} = k.$
%\end{eg*}

The following gives the precise description of this choice in general.

\begin{Def} \label{D:P}
Let $\kappa = (\kappa_1,\kappa_2,\dotsc,\kappa_l)$ be the $p$-adic $p$-composition of $k$, and let $P_k$ be the Sylow $p$-subgroup of $\sym{pk}$ chosen as follows:
$$
P_k = P_k(1) \times P_k(2) \times \dotsb \times P_k(l),$$
where $P_k(i) = R_{\kappa_i}[\Sigma^{\kappa}_i/p^{\kappa_i}]$ for all $i$.
That is, the $i$-th factor $P_k(i)$ is a copy of $R_{\kappa_i}$ with support being translated from $(0,p^{\kappa_i+1}]$ to $(p\Sigma^{\kappa}_{i-1}, p\Sigma^{\kappa}_i]$.
\end{Def}

\begin{note}
Note that $\pi^{[p^i]}[j] \in P_k$ if and only if $jp^i \leq k$.  In fact,
$$
P_k = \langle \pi^{[p^i]}[j] \mid jp^i \leq k \rangle.
$$
This may be used as an alternative definition of $P_k$.
\end{note}

\begin{notation}
In what follows, we will frequently have expressions of the form
$\prod_{i=1}^l x_i[\Sigma^{\kappa}_i/p^{\kappa_i}]$ and $\prod_{i=1}^l A_i[\Sigma^{\kappa}_i/p^{\kappa_i}]$ where for each $i \in [1,l]$, $x_i$ and $A_i$ are respectively an element and a subset of $\sym{p^{\kappa_i+1}}$. Most of the time, the details of the shifts do not play a role.  We will therefore use the
shorthand notations
$$\prod_{\kappa} x_i \text{\quad and \quad} \prod_{\kappa} A_i
$$
to denote these expressions.

Similarly if $\lambda = (\lambda_1, \ldots, \lambda_s)$
is a $p$-composition then we write
$$\prod_{\lambda} x_i = \prod_{i=1}^s x_i[\Sigma_i^{\lambda}/p^{\lambda_i}]
$$
(where $x_i \in \sym{p^{\lambda_i+1}}$).
\end{notation}

\begin{lem}\label{L:Pn}
Let $P_k$ be the Sylow $p$-subgroup of $\sym{kp}$ defined above. Then we have $P_k \subseteq P_{k+1}$.
\end{lem}

\begin{proof}
This follows from the fact that
$
P_k = \langle \pi^{[p^i]}[j] \mid jp^i \leq k \rangle.
$
\end{proof}

\begin{comment}
\begin{proof}  If $p$ does not divide $n+1$ then $P_n=P_{n+1}$.
So assume now that $p$ divides $n+1$.
Assume first that
$n = (p-1)p^j + (p-1)p^{j-1} + \ldots + (p-1)p + (p-1)$ so that
$n+1 = p^{j+1}$ and $P_{n+1} = R_j$.
We proceed by induction on $j$, the case  $j=0$ is clear.
Assume the statement is  for $j-1\geq 0$.
Write $P_n= H\times K$ where $H$ is a product of $(p-1)$ copies
of $R_{j-1}$ with disjoint supports, and where $K$ is
$P_{n'}$ shifted by $(p-1)p^j$, where $n' = n-(p-1)p^j$.
By the inductive hypothesis, $P_{n'} \subset P_{n'+1}= R_{j-1}$, so
$K$ is contained in the shift of $R_{j-1}$.
This means that $P_n$ is contained in the product of $p$ copies of
$R_{j-1}$ with disjoint supports, that is, contained in $B_{j-1} \subset R_j$.

For the general case, write
$n = n_1 + n_2$ where $n_1 = r_tp^t + \ldots + r_{j+1}p^{j+1}$
with $r_{j+1}< p-1$ and $n_2 = (p-1)p^j +(p-1) p^{j-1} + \ldots + (p-1)$
Write $P_n = P_{n_1}\times K$ with disjoint supports, where $K$ is
$P_{n_2}$ shifted by $n_1$. By the first case, $P_{n_2}\subset P_{n_2+1}$
and then $K$ is contained $R_{j-1}$ shifted by $n_1$. Therefore $P$ is
contained in $P_{n+1}$.
\end{proof}
\end{comment}

For each $g \in P_k$, we have $g = \prod_{\kappa} \rho_i$ for some $\rho_i \in R_{\kappa_i}$ for all $i \in [1,l]$.  By Proposition \ref{P:wreath}, each $\rho_i$ is associated with a $p$-composition $\lambda^{(i)}$ of $p^{\kappa_i+1}$ arising from the finest $\rho_i$-invariant partition of the interval $(0,p^{\kappa_i+1}]$.  The concatenation $\lambda$ of these $p$-compositions is again a $p$-composition, of $pk$, by Lemma \ref{L:refinement} (note that $(\kappa_1+1,\kappa_2 + 1,\dotsc, \kappa_l+1)$ is the $p$-adic $p$-composition of $pk$).

If in addition $g$ is fixed-point-free as an element of $\sym{pk}$, then the parts in each $\lambda^{(i)}$ are positive.  Subtracting each part in $\lambda^{(i)}$ by 1 produces a $p$-composition $\bar{\lambda}^{(i)}$ of $p^{\kappa_i}$, and the concatenation $\bar{\lambda}$ of these $\bar{\lambda}^{(i)}$ is a $p$-composition of $k$.  We call $\bar{\lambda}$ the $p$-composition of $k$ associated to the fixed-point-free element $g$ of $P_k$.

The following shows that every $p$-composition of $k$ arises in this way.  In fact, each is associated to some fixed-point-free element $g$ of $P_k$ of order $p$.

\begin{prop} \label{P:p-comp-2}
Let $k \in \mathbb{Z}^+$, and let $\lambda = (\lambda_1,\dotsc, \lambda_s)$ be a $p$-composition of
$k$. Then $\lambda$ is associated to some fixed-point-free element $g \in P_k$ of order $p$.
\end{prop}

\begin{proof} \
%Let $\lambda = (\lambda_1, \ldots,\lambda_s)$ be a $p$-composition
%of $n$.
If $s=1$, then we apply Proposition
\ref{P:p-comp-1}.
So suppose $s>1$ and let $r = \Sigma_{s-1}^{\lambda}$ ($= k - p^{\lambda_s}$). By induction
$(\lambda_1, \ldots,
\lambda_{s-1})$ is associated to some $g_1\in P_r$ of order $p$ which is fixed-point-free on $(0,pr]$.
Since $\pi^{[p^{\lambda_s}]}[k/p^{\lambda_s}] \in P_k$, is transitive and fixed-point-free on $(pr,pk]$, we see that
$g = g_1 \pi^{[p^{\lambda_s}]}[k/p^{\lambda_s}] \in P_k$ has order $p$, is fixed-point-free on $(0,pk]$, with associated $p$-composition $\lambda$.
\end{proof}

\begin{prop} \label{P:refineconverse}
Let $k \in \mathbb{Z}^+$, with $p$-adic $p$-composition $\kappa = (\kappa_1,\dotsc, \kappa_l)$.  Then every $p$-composition of $k$ is a refinement of $\kappa$.
\end{prop}

\begin{proof}
Let $\lambda = (\lambda_1,\dotsc, \lambda_s)$ be a $p$-composition of $k$.  By Proposition \ref{P:p-comp-2}, $\lambda$ is associated to a fixed-point-free element $g \in P_k$ of order $p$.  Equivalently, $\{0,p\Sigma^\lambda_1,p\Sigma^\lambda_2, \dotsc, p\Sigma^\lambda_s\}$ is the finest $g$-invariant partition of the interval $(0,pk]$.  But since $P_k \subseteq \prod_{\kappa} \sym{p^{\kappa_i+1}}$, we see that $\{ 0, p\Sigma^{\kappa}_1, p \Sigma^{\kappa}_2, \dotsc, p\Sigma^{\kappa}_l\}$ is a $g$-invariant partition of the interval $(0,pk]$.  Thus, $\{0,p\Sigma^\lambda_1,p\Sigma^\lambda_2, \dotsc, p\Sigma^\lambda_s\}$ is finer than $\{ 0, p\Sigma^{\kappa}_1, p \Sigma^{\kappa}_2, \dotsc, p\Sigma^{\kappa}_l\}$, and hence $\{0,\Sigma^\lambda_1,\Sigma^\lambda_2, \dotsc, \Sigma^\lambda_s\}$ is finer than $\{ 0, \Sigma^{\kappa}_1,  \Sigma^{\kappa}_2, \dotsc, \Sigma^{\kappa}_l\}$, which in turn shows that $\lambda$ is a refinement of $\kappa$.
\end{proof}

\section{Finding right  coset representatives} \label{S:exist}

We will now analyse the right cosets $Dx$ such that $(\Delta_k \sym{p})^x\cap P\neq 1$; from now on, $P = P_k$, as defined in Definition \ref{D:P}.
Our aim in this section is to find a good subset $X_k \subseteq \sym{pk}$ such that
\begin{itemize}
\item $(\Delta_k \sym{p})^{y}\cap P\neq 1$ for all $y \in X_k$;
\item if $(\Delta_k \sym{p})^x\cap P\neq 1$, then there exists $y \in X_k$ such that $Dx = Dy$.
\end{itemize}

When $(\Delta_k \sym{p})^x\cap P\neq 1$, this is a non-trivial $p$-subgroup of $P$ conjugate to a subgroup of
$\Delta_k \sym{p}$, and hence it is generated by a fixed-point-free element $y$ of $P$ of order $p$.
Clearly we can replace $x$ by elements of the form $dx$ with $d\in D$ without altering the right coset $Dx$. Taking $d$ suitably in $\Delta_k\sym{p} \subseteq D$ will allow us to have $y=(\Delta_k \pi)^x$ in $P$.
Such $x$ takes orbits of $\Delta_k \pi$ to orbits of $y$, and the order in which these orbits
appear can be controlled by modifying with $d\in \sym{k}^{[p]}$. The following makes this
precise.

\begin{prop} \label{P:prelim}
Let $\kappa$ and $P$ be as  in Definition \ref{D:P}.  Let $x \in \sym{pk}$ such that $(\Delta_k \sym{p})^x \cap P \ne 1$.  Then for each
$r\in [1,l]$ there exists $x_r \in \sym{p^{\kappa_r+1}}$  such that
\begin{itemize}
\item $\prod_{\kappa} x_j \in Dx$;
\item $(\Delta_{p^{\kappa_r}} \pi)^{x_r} \in R_{\kappa_r}$.
\end{itemize}
\end{prop}

\begin{proof}
Let $1 \ne y \in (\Delta_k \sym{p})^x \cap P$, say $y = (\Delta_k g)^x$ for some $p$-cycle $g \in \sym{p}$. Then $g = \pi^{\tau}$ for some
$\tau \in \sym{p}$; and by replacing $x$ with $(\Delta_k \tau)  x \in Dx$
we may assume $g= \pi$.

Since $y \in P = P(1) \times \dotsb \times P(l)$, we have $y = \prod_{r=1}^l y_r$ where for each $r$, $y_r \in P(r) = R_{\kappa_r}[\Sigma^\kappa_r/p^{\kappa_r}]$.  As $
\prod_{r=1}^l y_r =  (\Delta_k \pi)^x = \prod_{i=1}^k (\pi[i])^x$, there exists $\sigma \in \sym{k}$ permuting the cycles of
$\Delta_k\pi$ such that, for each $r \in [1,l]$, we have
$\prod_{i = \Sigma_{r-1}^{\kappa}+1}^{\Sigma_r^{\kappa}} \pi[i\sigma]^x = y_{r}$. Recall that $\prod_{i=1}^k\pi[i\sigma] = (\Delta_k\pi)^{\sigma^{[p]}}$, so
by  replacing $x$ with $\sigma^{[p]} x \in Dx$, we may assume that
$$
((\Delta_{p^{\kappa_r}} \pi)[\Sigma^{\kappa}_r/p^{\kappa_r}])^x =
%\prod_{i = \Sigma_{r-1}^{\kappa}+1}^{\Sigma_r^{\kappa}} \pi[i]^x =
y_r \in P(r) = R_{\kappa_r}[\Sigma^\kappa_r/p^{\kappa_r}]
 $$
for all $r \in [1,l]$.  Thus $x$ preserves the orbits of $P$ and we
can write  $x = \prod_{r=1}^l z_r$, where $z_r \in \sym{p^{\kappa_r+1}}[\Sigma^\kappa_r/p^{\kappa_r}]$ for all $r$, and
$y_r =
%((\Delta_{p^{\kappa_r}} \pi)[\Sigma^{\kappa}_r/p^{\kappa_r}])^x =
((\Delta_{p^{\kappa_r}} \pi)[\Sigma^{\kappa}_r/p^{\kappa_r}])^{z_r}$.  The proposition follows by defining $x_r$ to be the element of $\sym{p^{\kappa_r+1}}$ such that $x_r[\Sigma^\kappa_r/p^{\kappa_r}] = z_r$.
\end{proof}

This shows in particular that we can take $x$ so that it respects the direct factors of $P$.
We have the following refinement, which concentrates on a fixed factor of $P$. The proof is analogous to that of Proposition \ref{P:prelim}.

\begin{prop} \label{P:prelim2}
Suppose that $(\Delta_{p^m} \pi)^x \in R_m$ for some $x \in \sym{p^{m+1}}$.  Then there exist a $p$-composition $\lambda = (\lambda_1,\lambda_2,\dotsc, \lambda_s)$ of $p^m$, and elements $x_r \in \sym{p^{\lambda_r+1}}$ for all $r \in [1,s]$ such that
\begin{itemize}
\item $\prod_{\lambda} x_r \in \sym{p^m}^{[p]}\, x$;
\item $(\Delta_{p^{\lambda_r}} \pi)^{x_r} \in R_{\lambda_r} \setminus B_{\lambda_r}$.
\end{itemize}
\end{prop}

\begin{proof}
By Proposition \ref{P:wreath} and Proposition \ref{P:p-comp-1}, there exist a $p$-composition $\lambda = (\lambda_1, \lambda_2, \dotsc, \lambda_s)$ of $p^m$, and elements $\gamma_r \in R_{\lambda_r} \setminus B_{\lambda_r}$ for $r \in [1,s]$, such that $(\Delta_{p^m} \pi)^x
= \prod_{\lambda}\gamma_r$.  Thus there is some
$\sigma \in \sym{p^m}$ such that $(\prod_{i= \Sigma^\lambda_{r-1} + 1}^{\Sigma^\lambda_r} \pi[i\sigma])^x = \gamma_r[\Sigma^\lambda_r/p^{\lambda_r}]$ for all $r\in [1,s]$.
By replacing $x$ by $\sigma^{[p]} x \in \sym{p^m}^{[p]}\, x$, we may assume that
$$
((\Delta_{p^{\lambda_r}} \pi)[\Sigma^\lambda_r/p^{\lambda_r}])^x
%(\prod_{i= \Sigma^\lambda_{r-1} + 1}^{\Sigma^\lambda_r} \pi[i])^x
= \gamma_r[\Sigma^\lambda_r/p^{\lambda_r}]
$$
for all $r$.  Thus $x = \prod_{r=1}^s z_r$, where $z_r \in \sym{p^{\lambda_r+1}}[\Sigma^\lambda_r/p^{\lambda_r}]$, and
$$\gamma_r[\Sigma^\lambda_r/p^{\lambda_r}] =
%((\Delta_{p^{\lambda_r}} \pi)[\Sigma^{\lambda}_r/p^{\lambda_r}])^x =
((\Delta_{p^{\lambda}} \pi)[\Sigma^{\lambda}_r/p^{\lambda_r}])^{z_r}.$$
The Proposition follows by defining $x_r$ to be the element of $\sym{p^{\lambda_r+1}}$ such that $x_r[\Sigma^\lambda_r/p^{\lambda_r}] = z_r$.
\end{proof}

We now find good coset representatives for these $x_r$, and we
start by defining  distinguished elements which will conjugate $\Delta_{p^m}\pi$ to
$(\pi^{[p^m]})^t$.

\begin{Def} \label{D:z}
For $t \in \mathbb{Z}_p^*$ and $m \in \mathbb{Z}_{\geq 0}$, define  $z_{m,t} \in \sym{p^{m+1}}$ by
$$((i-1)p+j)z_{m,t} := i+ (\overline{t(j-1)})p^m \quad (i\leq p^m, j\leq p).
$$
Here, and hereafter, given an integer $x$, we write $\overline{x}$ for the residue of $x \pmod p$ with $0\leq \overline{x} < p$.
\end{Def}

For example, if $m=0$ then $z_{0,t}$ normalizes the group $\langle \pi\rangle$ (see Lemma \ref{L:z} below), and $z_{0,1} = 1$.

%the unique element satisfying $(\pi[i])^{z_{m,t}} = (i,i+p^m, \dotsc, i+(p-1)p^m)^t$ and $((i-1)p+1)z_{m,t} = i$ for all $i \in \{ 1, 2, \dotsc, p^m \}$.

\begin{note} \label{N:tableaux}
It will be convenient to describe $z_{m,t}$ by making use of the natural faithful action of $\sym{n}$ on standard tableaux.
We take tableaux with $p^m$ rows and each row of length $p$, i.e.\ of
shape $(p^{p^m})$.  Then
$z_{m,1}$ is the element of $\sym{p^m+1}$ which sends the tableau
    $$
    \begin{smallmatrix}
    1 & 2 & \dotsb & p \\
    p+1 & p+2 & \dotsb & 2p \\
    \vdots & \vdots & & \vdots \\
    p^m -p + 1 & p^m -p + 2 & \dotsb & p^{m+1}
    \end{smallmatrix} \quad{\text{to}} \qquad
    \begin{smallmatrix}
    1 & p^m+1 & \dotsb & (p-1)p^m + 1 \\
    2 & p^m+2 & \dotsb & (p-1)p^m + 2 \\
    \vdots & \vdots & & \vdots \\
    p^m & 2p^m & \dotsb & p^{m+1}
    \end{smallmatrix}\ .$$
We denote the first tableau by $T$, and the second one by $\tilde{T}$, and we
will use these notations later.

In general, $z_{m,t}$ sends $T$ to the tableau
    $$
    \begin{smallmatrix}
    1 & \overline{t}p^m+1 & \dotsb & (\overline{t(p-1)})p^m + 1 \\
    2 & \overline{t}p^m+2 & \dotsb & (\overline{t(p-1)})p^m + 2 \\
    \vdots & \vdots & & \vdots \\
    p^m & (\overline{t}+1)p^m & \dotsb & (\overline{t(p-1)+1})p^{m}
    \end{smallmatrix}\ $$
which is obtained from $\tilde{T}$ by permuting the columns according to
$z_{0,t}$.
\end{note}

\begin{lem}\label{L:z} Let $z_{m,t}$ be as above.
\begin{enumerate}
\item
We have $\pi[i]^{z_{m,t}} = (i, p^m +i, \dotsc, (p-1)p^m + i)^t$ for $i\in [1, p^m]$.  Thus,
$(\Delta_{p^m} \pi)^{z_{m,t}} = (\pi^{[p^m]})^t$.

\item
For $s, t \in \Z_p^*$ we have
$(\Delta_{p^m}z_{0,t})\cdot z_{m,s} = z_{m, ts}$.
\end{enumerate}
\end{lem}

\begin{proof}
For the first part, consider the tableaux $T$ and $T z_{m,t}$ above.
The rows of $T$ are the cycles
of $\Delta_{p^m} \pi$, and the rows of $Tz_{m,t}$ are the cycles of $(\pi^{[p^m]})^t$. The
standard formula for conjugation gives the statement.
The second part follows from a direct verification using the definition of $z_{m,t}$.
\end{proof}

%We still concentrate on a factor $R_m$ of $P$. When $\Delta_{p^m}\pi$ is  conjugated, we need
%to take care  of its centraliser in $\sym{p^{m+1}}$. This is the group
%$H_{p^m}\rtimes \sym{p^m}^{[p]}$ (and this explains
%why in the following the group $\sym{p^m}^{[p]}$ appears).
%In particular we should  expect factors from $H_{p^m}$ occuring as well.

We denote by $R_m^0$ the subgroup of $B_m$ consisting of elements which fix
the last block of $p^m$ elements pointwise, i.e.\
$$R_m^0 = \prod_{i=1}^{p-1} R_{m-1}[i]$$

\begin{prop} \label{P:prelim3}
Suppose that $(\Delta_{p^m} \pi)^x \in R_m \setminus B_m$ for some $x \in \sym{p^{m+1}}$.  Then
there exist unique $t \in \mathbb{Z}_p^*$, $h \in H_{p^m}$ and $b \in R_m^0$
such that
$$
hz_{m,t}b \in \sym{p^{m}}^{[p]}\,x.
$$
\end{prop}

\begin{proof}
Since $(\Delta_{p^m} \pi)^x$ lies in $R_m \setminus B_m$ and has order $p$,
we apply Corollary \ref{C:conj}. Hence there exist unique $t \in \mathbb{Z}_p^*$ and $b \in R_m^0$ such that $(\Delta_{p^m} \pi)^x = ((\pi^{[p^m]})^t)^b$.
Thus
$$ (\Delta_{p^{m}} \pi)^x = ((\pi^{[p^{m}]})^t)^b = (\Delta_{p^{m}} \pi)^{z_{m,t}b},$$
and so $z_{m,t}b x^{-1}$ lies in the centraliser of $\Delta_{p^{m}} \pi$ in $\sym{p^{m+1}}$,
which is  $H_{p^m} \rtimes \sym{p^m}^{[p]}$.  Hence, there exists a unique $h \in H_{p^m}$ such that $h z_{m,t}b \in \sym{p^m}^{[p]}\, x$.
\end{proof}

%Combining Propositions \ref{P:prelim2} and \ref{P:prelim3}, we get the following:

%\begin{cor} \label{C:prelim4}
%Suppose that $(\Delta_{p^m} \pi)^x \in R_m$ for some $x \in \sym{p^{m+1}}$.  Then there exist a $p$-composition $\lambda = (\lambda_1,\lambda_2,\dotsc, \lambda_s)$ of $p^m$, $t_1,\dotsc, t_s \in \mathbb{Z}_p^*$ and $h \in H_{p^m}$ such that
%$$h (\prod_{r=1}^s z_{\lambda_r, t_r}[\Sigma^{\lambda}_r/p^{\lambda_r}]) \in (\sym{p^m})^{[p]}\, x B_m.$$
%\end{cor}

Propositions \ref{P:prelim}, \ref{P:prelim2} and \ref{P:prelim3} suggest the following definition for the desired coset representatives.

\begin{Def}\label{D:X} \  For $m\in \Z_{\geq 0}$, we define the subset
$Y_m$ of $\sym{p^{m+1}}$ to be
$$Y_m:= \{ hz_{m,t}b \mid h\in H_{p^m}, t\in \Z_p^*, b\in R_m^0\}
$$
For example, $Y_0 = \{ hz_{0,t}: h\in \langle \pi \rangle, t\in \Z_p^*\}$,
which is the normalizer of $\langle \pi\rangle$ in $\sym{p}$.
For a $p$-composition $\lambda= (\lambda_1, \lambda_2, \ldots, \lambda_s)$ of
$d$, define
$$X_{\lambda}:= \prod_{\lambda} Y_{\lambda_r} \qquad { and } \qquad
X_d:= \bigcup_{\lambda} X_{\lambda}.
$$
\end{Def}

We also have a recursive description.

\begin{lem} \label{L:X} \hfill
\begin{enumerate}
\item Let $m \in \mathbb{Z}_{\geq 0}$. Then $$X_{p^m} = Y_m \cup \prod_{i=1}^p X_{p^{m-1}}[i]$$
(disjoint union, and where $X_{p^{-1}} = \emptyset$).
\item Let $k \in \mathbb{Z}^+$, with $p$-adic $p$-composition $\kappa = (\kappa_1,\dotsc,\kappa_l)$.  Then $$X_k = \prod_{\kappa} X_{p^{\kappa_i}}.$$
%\item Let $k \in \mathbb{Z}^+$.  Then $(\Delta_k z_{0,t}) X_k = X_k$.
\end{enumerate}
\end{lem}

\begin{proof} This follows from Lemma \ref{L:refinement}, Corollary \ref{C:refinement}
and Proposition \ref{P:refineconverse}.
\end{proof}

We can now state the main theorem of this section.

\begin{thm} \label{T:exist}
Let $P$ be as in Definition \ref{D:P}.  Let $x \in \sym{pk}$ such that
$(\Delta_k\sym{p})^x \cap P \ne 1$.  Then there exists $x_0 \in X_k$ such that $x_0 \in Dx$.
\end{thm}

\begin{proof}
This follows from Propositions \ref{P:prelim}, \ref{P:prelim2} and \ref{P:prelim3}, and Lemma \ref{L:X}(2).%\footnote{Is this sufficient?  Or more details?}.
\end{proof}

We note that the converse of Theorem \ref{T:exist} also holds.

\begin{lem}
We have $(\Delta_k\sym{p})^y \cap P \ne 1$ for all $y \in X_k$.
\end{lem}

\begin{proof}
By the definition of the $z_{\lambda_r, t_r}$, the conjugate $(\Delta_k \pi)^y$
belongs to $P$.
\end{proof}

\section{Uniqueness} \label{S:unique}

We have seen in the previous section that $(\Delta_k\sym{p})^y \cap P \ne 1$ for all
$y \in X_k$, and if $(\Delta_k\sym{p})^x \cap P \ne 1$,
then there exists $x_0 \in X_k$ such that $Dx_0  = Dx$.
However, $x_0$ need not be unique.  For example, instead of $x_0$, one may choose $(\Delta_k \pi)x_0$ which also lies in $X_k$.

In this section, we will show that, when $p \nmid k$, $X_{k-1}$ is a
transversal for the right cosets $Dx$ satisfying
$(\Delta_k\sym{p})^x\cap P\neq 1$.
We begin with the observation that a transversal for these right cosets can be chosen to be a subset of $X_{k-1}$.

\begin{prop} \label{P:reduction0}
Let $k \in \mathbb{Z}^+$ with $p \nmid k$.  Then $X_{k-1} \subseteq X_k$.  Furthermore, if $x \in X_k$, then there exists $y \in X_{k-1}$ such that $Dx = Dy$.
\end{prop}

\begin{proof}
Let $\kappa = (\kappa_1,\kappa_2,\dotsc, \kappa_l)$ be the $p$-adic $p$-composition of $k$.  Since $p \nmid k$, we see that $\kappa_l = 0$, and $(\kappa_1,\dotsc, \kappa_{l-1})$ is the $p$-adic $p$-composition of $(k-1)$.  Thus, $X_k = X_{k-1} \times X_1[k]$ by Lemma \ref{L:X}(2).  The first assertion now follows as $1$ ($= z_{0,1}$) belongs to $X_1$.

For the second assertion, we have $x = a \cdot x'[k]$ where $a \in X_{k-1}$ and $x' = h z_{0,t}$ with $h \in H_1$ and $t \in \mathbb{Z}_p^*$.  Let $u$ be the inverse of $t$ in $\mathbb{Z}_p^*$.  Then $z_{0,u}$ is the inverse of $z_{0,t}$ by Lemma \ref{L:z}(2).  Define  $y: = \Delta_{k-1}(z_{0,u}h^{-1}) a$.
Then $y \in Dx$ since $\Delta_k(z_{0,u}h^{-1})\in D$ and
$$
\Delta_k (z_{0,u} h^{-1}) x = \Delta_{k-1} (z_{0,u}h^{-1})a \cdot (z_{0,u}h^{-1} x')[k] = y.$$
Furthermore, $y\in X_{k-1}$; to see this, note that $\Delta_{k-1} z_{0,u}$ normalises $H_{k-1}$ and
use Lemma \ref{L:z}(2).
\end{proof}

Before we continue, we make the following observation:

\begin{lem}\label{L:groups}  Let $m \in \mathbb{Z}_{\geq 0}$, and let $k \in \mathbb{Z}^+$ with $p$-adic $p$-composition $(\kappa_1,\dotsc,\kappa_l)$.  Then
\begin{enumerate}
\item $\sym{p^m}^{[p]} \cap \prod_{i=1}^p \sym{p^m}[i] = \prod_{i=1}^p
\sym{p^{m-1}}^{[p]}[i]$;
\item $\sym{k}^{[p]} \cap \prod_{\kappa} \sym{p^{\kappa_i+1}} = \prod_{\kappa}
\sym{p^{\kappa_i}}^{[p]}$.
\end{enumerate}
\end{lem}

\begin{proof}
Note first that $\sym{p^m}^{[p]}$ consists {\it precisely} of the permutations of $\sym{p^{m+1}}$
which, in the natural action, induce row permutations on the tableau
$T$.  Suppose that $\sigma^{[p]} \in \prod_{i=1}^p \sym{p^m}[i]$ for some $\sigma \in \sym{p^m}$.
Then $\sigma^{[p]}$ leaves each successive block in $[1,p^{m+1}]$ of size $p^m$ invariant.  Thus, on the tableau $T$, $\sigma^{[p]}$ leaves each of the $p$ sub-tableaux consisting of $p^{m-1}$ successive rows of $T$ invariant setwise, so that $\sigma^{[p]} \in \prod_{i=1}^p \sym{p^{m-1}}^{[p]}[i]$.  This shows $\sym{p^m}^{[p]} \cap \prod_{i=1}^p \sym{p^m}[i] \subseteq \prod_{i=1}^p
\sym{p^{m-1}}^{[p]}[i]$.  The converse clearly holds, since $\sym{p^{m-1}}^{[p]} \subseteq \sym{p^m}$ and $$\prod_{i=1}^p \sym{p^{m-1}}^{[p]}[i] = (\prod_{i=1}^p \sym{p^{m-1}}[i])^{[p]} \subseteq \sym{p^m}^{[p]}.$$
This proves part (1).  Part (2) is similar.
\end{proof}

\begin{prop} \label{P:reduction}
Let $k \in \mathbb{Z}^+$ such that $p \nmid k$, with $p$-adic $p$-composition $\kappa = (\kappa_1,\kappa_2,\dotsc, \kappa_l)$.  For each $i \in [1,l)$, let $x_i, y_i \in X_{p^{\kappa_i}}$.  Let $x_l=y_l=1$, and let $x = \prod_{\kappa} x_i$ and $y = \prod_{\kappa}y_i$.  The following statements are equivalent:
\begin{enumerate}
\item $\sym{p^{\kappa_i}}^{[p]} x_i = \sym{p^{\kappa_i}}^{[p]} y_i$ for all $i \in [1,l)$;
\item $\sym{k}^{[p]}x = \sym{k}^{[p]} y$;
\item $Dx = Dy$.
\end{enumerate}
\end{prop}

\begin{proof} \ $(1) \Rightarrow (2)$ and $(2) \Rightarrow (3)$
follow from the fact that
$$\prod_{\kappa} \sym{p^{\kappa_i}}^{[p]} \subseteq \sym{k}^{[p]} \subseteq D.$$

$(3) \Rightarrow (1)$: \
Suppose that $Dx = Dy$.  Then there exist $\sigma \in \sym{k}$ and $\tau \in \sym{p}$ such that
\begin{equation*}
\sigma^{[p]}(\Delta_k \tau) \prod_{\kappa}x_i = \prod_{\kappa} y_i. \tag{$*$}
\end{equation*}
This gives
$$
\sigma^{[p]} = \prod_{\kappa} (y_i x_i^{-1} (\Delta_{p^{\kappa_i}} \tau)^{-1})\in \prod_{\kappa} \sym{p^{\kappa_i+1}},$$
so that $\sigma^{[p]} \in \prod_{\kappa} \sym{p^{\kappa_i}}^{[p]}$ by Lemma \ref{L:groups}(2).
Thus, for each $i\in [1,l]$  there exists $\sigma_i \in \sym{p^{\kappa_i}}$
such that $\sigma^{[p]} = \prod_{\kappa} \sigma_i^{[p]}$.  Putting this into $(*)$, we get, for all $i \in [1,l]$,
$$
\sigma_i^{[p]}(\Delta_{p^{\kappa_i}} \tau) x_i = y_i.
$$
When $i = l$, we have $\sigma_l = 1$ since $\kappa_l=0$, and hence $\tau = 1$, since $x_l=y_l=1$.  Thus, we have, for $i \in [1,l)$,
$$
y_i = \sigma_i^{[p]} x_i \in \sym{p^{\kappa_i}}^{[p]} x_i.
$$
\end{proof}

In view of Proposition \ref{P:reduction}, our problem reduces to determining the necessary and sufficient conditions for $\sym{p^m}^{[p]}x = \sym{p^m}^{[p]}y$ where $x,y \in X_{p^m}$.

Recall that $X_{p^m}$ is a disjoint union of $Y_m$ and $\prod_{i=1}^p X_{p^{m-1}}[i]$ (Lemma \ref{L:X}(1)).  We consider these two sets separately.

\begin{prop} \label{P:reduction2}
Let $m \in \mathbb{Z}^+$ and for each $i \in [1,p]$, let $x_i, y_i \in X_{p^{m-1}}$.  The following statements are equivalent:
\begin{enumerate}
\item $\sym{p^m}^{[p]} (\prod_{i=1}^p x_i[i]) = \sym{p^m}^{[p]} (\prod_{i=1}^p y_i[i])$.
\item $\sym{p^{m-1}}^{[p]} x_i = \sym{p^{m-1}}^{[p]} y_{i}$ for all $i \in [1,p]$.
\end{enumerate}
\end{prop}

The proof of this is straightforward, using Lemma \ref{L:groups}(1).
%\begin{proof} \hfill
%\begin{description}
%\item[$(1) \Rightarrow (2)$]  We have $\prod_{i=1}^p (x_iy_i^{-1})[i] \in \sym{p^m}^{[p]} \cap \prod_{i=1}^p \sym{p^m}[i] = \prod_{i=1}^p \sym{p^{m-1}}^{[p]}[i]$ by Lemma \ref{L:groups}(1), so that $x_iy_i^{-1} \in \sym{p^{m-1}}^{[p]}$ for all $i \in [1,p]$.
%\item[$(2) \Rightarrow (1)$]  This follows since $\prod_{i=1}^p \sym{p^{m-1}}^{[p]}[i] \subseteq \sym{p^m}^{[p]}$ by Lemma \ref{L:groups}(1).
%\end{description}
%\end{proof}
It remains to consider the right cosets $\sym{p^m}^{[p]} x$ where $x \in Y_m$.

\begin{prop} \label{L:3}  Let $m \in \mathbb{Z}_{\geq 0}$.  Suppose that there exist $x \in X_{p^m}$ and $y \in Y_m$ such that
$\sym{p^m}^{[p]}x = \sym{p^m}^{[p]}y$. Then $x = y$.
\end{prop}

\begin{proof}
We have $y = \tau^{[p]} x$ for some $\tau \in \sym{p^m}$.  Since $y \in Y_m$, we have $(\Delta_{p^m} \pi)^y \in R_m \setminus B_m$, and since
$\tau^{[p]}$ centralises $\Delta_{p^m}\pi$ we have
$$(\Delta_{p^m} \pi)^x  = (\Delta_{p^m} \pi)^y \in R_m \setminus B_m.$$
We claim that $x\in Y_m$. If not, then
$x \in \prod_{i=1} X_{p^{m-1}}[i] \subseteq \prod_{i=1}^p \sym{p^{m}}[i]$, say $x = \prod_{i=1}^p x_i[i]$, and then $(\Delta_{p^m} \pi)^x = \prod_{i=1}^p (\Delta_{p^{m-1}} \pi)^{x_i}[i] \in \prod_{i=1}^p \sym{p^{m}}[i]$,
a contradiction since $(R_m \setminus B_m) \cap \prod_{i=1}^p \sym{p^{m}}[i] =\emptyset$.  Thus $x \in Y_m$ and hence $x = y$ by the uniqueness result
of Proposition \ref{P:prelim3}.
\end{proof}

\begin{thm} \label{T:unique}
Let $k \in \mathbb{Z}^+$ with $p \nmid k$.  Then $X_{k-1}$ is a transversal of the right cosets $Dx$ satisfying  $(\Delta_k \sym{p})^x \cap P \ne 1$.
%That is, if the $p$-adic expansion of $k$ is $k=\sum_{i=0}^t b_ip^i$ then
%the set $X_{k-1}$ is the disjoint product of
%$b_i$ copies of $X_{p^i}$ for $i\geq 1$, and of $b_0-1$ copies of
%$X_{p^0}$.
\end{thm}

\begin{proof}
This follows from Propositions \ref{P:reduction0}, \ref{P:reduction}, \ref{P:reduction2} and \ref{L:3}.
\end{proof}

\begin{cor} \hfill \label{C:unique}
\begin{enumerate}
\item Let $m \in \mathbb{Z}_{\geq 0}$, and let $a_m = |X_{p^m}|$.  Then $a_m$ satisfies the following recurrence relation (where $a_{-1} = 0$):
$$
a_{m} = a_{m-1}^p + p^{2p^m-1}(p-1).
$$
\item Let $k \in \mathbb{Z}^+$ such that $p \nmid k$, with $p$-adic $p$-composition $(\kappa_1,\dotsc, \kappa_l)$.  Then $|X_{k-1}| = \prod_{i=1}^{l-1} a_{\kappa_i}$.  In particular,
    $$\dim ((\Res_P \Lie(kp))_{pf}) = (p-1)(k-1)! \prod_{i=1}^{l-1} a_{\kappa_i}.$$
\end{enumerate}
\end{cor}

\begin{proof}
From the uniqueness of Proposition \ref{P:prelim3}, we have $$|Y_m| = |H_{p^m}||\mathbb{Z}_p^*||R_{m-1}|^{p-1} = p^{2p^m-1}(p-1)$$ (note that $|R_{m-1}| = p^{\frac{p^m-1}{p-1}}$).  The Corollary follows now
from Lemma \ref{L:X}, Theorem \ref{T:unique} and Corollary \ref{C:main}.
\end{proof}

\begin{thm} \label{T:expgrowth}
Let $k \in \mathbb{Z}^+$ with $p \nmid k$.  Then $e^{C_1k} \leq |X_{k-1}| \leq e^{C_2k}$ is for some constants $C_1 < C_2$.

In particular, the dimension of $(\Res_P\Lie(kp))_{pf}$ grows exponentially with $k$, but $\dim((\Res_P\Lie(kp))_{pf})/ \dim(\Lie(kp)) \to 0$ as $k \to \infty$.
\end{thm}

\begin{proof}
Let $\kappa= (\kappa_1,\dotsc,\kappa_l)$ be the $p$-adic $p$-composition of $k$.  Then $\kappa$ is a $p$-regular partition (i.e.\ it is weakly decreasing and does not have $p$ or more equal parts), and $|X_{k-1}| = \prod_{i=1}^{l-1} a_{\kappa_i}$ by Corollary \ref{C:unique}(2).  It is not difficult to show by induction that $a_{\kappa_1} \leq \prod_{i=1}^{l-1} a_{\kappa_i} \leq a_{\kappa_1 + 1}$, and that $a_m$ is of order $p^{2p^m}$.  Thus $e^{C_1k} \leq |X_{k-1}| \leq e^{C_2k}$ for some constants $C_1 < C_2$.

Since $\dim ((\Res_P \Lie(kp))_{pf}) = (p-1)(k-1)! |X_{k-1}|$, and $m!$ is of order $e^{m\log m}$, we see that $\dim (\Res_P\Lie(kp))_{pf}$ grows exponentially with $k$, and $\dim(\Res_P\Lie(kp))_{pf}/ \dim(\Lie(kp)) \to 0$ as $k \to \infty$.
\end{proof}

Theorem \ref{T:expgrowth} thus lends some support to the belief that $\Liemax(kp)$ is relatively large compared with $\Lie(kp)$, as mentioned in our introduction, even though $\Lie(kp)_{pf}$ grows exponentially with $k$.

%Namely, given a fixed prime $p$, then for large $k$ the size of
%$Y_m$ dominates where $p^m$ is the largest power of $p$ occuring
%in the $p$-adic expansion of $k$. This set has size $cp^{2m}$ where $c= (p-1)p^{-1}$ is a constant, that is the size %is $c\cdot e^{2m\ln(p)}$.

\begin{rem}
Selick and Wu \cite{SW2} computed explicitly $\Lie^{\max}(6)$ and $\Lie^{\max}(8)$ in characteristic two.  In particular, they showed that $\Lie^{\max}(6)$ has dimension $96$, which is also the upper bound computed by our recurrence formula.
\end{rem}

\end{document}